\documentclass[11pt]{amsart}
\usepackage{pdfsync}
\usepackage{graphicx}
\usepackage{amsmath}
\usepackage{amsfonts}
\usepackage{enumerate}
\usepackage{latexsym}
\usepackage{epsfig}
\usepackage{amsthm}  
\usepackage{amsmath}
\usepackage{amssymb} 
\usepackage{latexsym}
\usepackage{epsfig} 
\usepackage{amssymb,latexsym}
\usepackage[all]{xy}
\usepackage[
            unicode,
            breaklinks=true,
            colorlinks=true]{hyperref}
\usepackage[active]{srcltx}
\numberwithin{equation}{section}
\newtheorem{Theorem}{Theorem}[section]

\newtheorem{thm}[Theorem]{Theorem}

\newtheorem{lemma}[Theorem]{Lemma}
\newtheorem{definition}[Theorem]{Definition}

\newtheorem{remark}[Theorem]{Remark}
\newtheorem{example}[Theorem]{Example}

\newtheorem*{namedtheorem}{\theoremname}
\newtheorem{theoremL}{Theorem}

\newcommand{\theoremname}{testing}
\newcommand{\R}{{\mathbb R}}
\newcommand{\Z}{{\mathbb Z}}

\newcommand{\T}{{\mathbb T}}

\newcommand{\Mo}{\mathcal{M}_\mathbb{T}}
\newcommand{\Mt}{\widetilde{\mathcal{M}_\mathbb{T}}}
\newcommand{\Dtt}{\widetilde{\mathcal{D}_\mathbb{T}}}
\newcommand{\Dt}{\mathcal{D}_\mathbb{T}}

\DeclareMathOperator{\agl}{AGL}

\DeclareMathOperator{\vol}{vol}

\DeclareMathOperator{\AGL}{AGL}
\DeclareMathOperator{\ConvHull}{ConvHull}
\DeclareMathOperator{\length}{length}

\title[On the density function on moduli 
spaces of toric $4$\--manifolds]{On the density function\\ on moduli 
spaces of toric $4$\--manifolds}
\author{Alessio Figalli \,\,\,\,\,\,\,\,\,\,\,\,\,\,\,\,\, \'Alvaro Pelayo}

\date{}

\begin{document}
\maketitle

\date{}
\begin{abstract}
The optimal density function assigns to each symplectic toric manifold $M$ a number $0<d\leq 1$ obtained
by considering the ratio between the maximum volume of $M$ which can be filled by symplectically embedded 
disjoint balls and the total symplectic volume of $M$. In the toric version of this problem, $M$ is toric and the balls
need to be embedded respecting the toric action on $M$. The goal of this note is first to give a brief survey of the notion of toric 
symplectic manifold and the recent constructions of moduli space structure on them, and recall how to define a natural density function on this moduli space. 
Then we review previous works which explain how the study of the density function can be reduced to a problem in convex geometry, and use this correspondence to 
to  give a simple description of the regions of continuity of the maximal density function when the dimension is $4$.
\end{abstract}

\section{Introduction}

In symplectic topology the ball packing problem asks how much of the volume of a symplectic
manifold $(M,\omega)$ can be approximated by symplectically embedded disjoint balls, see Figure~\ref{embedding}.  
This is in general a very difficult problem; a lot of progress on it and directly related problems
has been made by a number of authors, among them Biran~\cite{B0,B1/2,B3}, Borman\--Li\--Wu~\cite{BoLiWu2013},
McDuff--Polterovich~\cite{MP}, Schlenk~\cite{Schlenk2005}, Traynor~\cite{T}, and Xu~\cite{Xu}. 
The \emph{optimal density function} $\Omega$ assigns to each closed symplectic manifold $M$ the number $0<d\leq 1$ obtained
by considering the ratio between the largest volume $v$ of $M$which can be filled by 
equivariantly and symplectically embedded disjoint balls, and the total symplectic volume ${\rm vol}(M)$ of $M$.
An  \emph{optimal} packing
is one for which the sum of these disjoint volumes divided by ${\rm vol}(M)$ is as large as it can be,
taking into consideration all possible such packings.

The article \cite{Pe2} discusses 
a particular instance of the symplectic ball packing problem: 
\emph{the  toric case}. In the toric case both the symplectic manifold $(M,\omega)$ and the standard
open symplectic ball $\mathbb{B}_r$ in $\mathbb{C}^n$ are equipped with a Hamiltonian action of
an $n$\--dimensional torus $\mathbb{T}$, and the symplectic embeddings 
of the ball into the manifold $M$ are equivariant with respect to these
actions.   In the case of symplectic toric manifolds,
there always exists at least an optimal packing \cite{PeSc2008}.  The best way to think of this
problem is an approximation theorem for our integrable system  
by disjoint integrable systems on balls. This problem is more rigid, because  for instance fixed
points of the system have to coincide with the origin of the ball, so the symplectic balls
mapped this way have less flexibility.

\begin{figure}[htb]
\begin{center}
\includegraphics{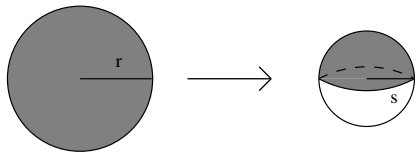}
\caption{A toric symplectic embedding of the $2$\--ball of radius $r$
into a $2$\--sphere of radius $s$ with $r/s=\sqrt{2}$.}\label{embedding}
\end{center}
\end{figure}

In this note we concentrate on symplectic toric manifolds (also called toric integrable systems)
because for this we have a good understanding \cite{Pe1,Pe2,PeSc2008}, but the question
is interesting for any integrable system. These particular systems are usually called \emph{toric}, or
\emph{symplectic toric}, and a rich structure theory due to Kostant, Guillemin\--Sternberg, Delzant among
others led to a complete classification in the 1980s  \cite{At1982, De1988, GuSt1982, Gu1994}. We refer
to Section~\ref{sec:st} for the precise definition of symplectic toric manifolds, and to 
Section~\ref{sec:moduli} for the definition of  the moduli
space $\mathcal{M}_{\T}$ of such manifolds, where $\T$ denotes the standard torus of dimension
precisely half the dimension of the manifolds in $\mathcal{M}_{\T}$. 
By using elementary
geometric arguments, we shall describe the regions of this moduli space where the optimal density
function $\Omega \colon \mathcal{M}_{\T} \to [0,1]$ is continuous. In particular, as a consequence of this description one can observe that
the density function is highly discontinuous.

The structure of the paper is as follows.  In Section \ref{sec:0} we state the main theorem of the paper: Theorem~\ref{symplectic-geometry}; in Section \ref{sec:st} we review the notion of symplectic toric manifold;
in Section \ref{sec:moduli} we review the construction of the moduli space $\Mo$ of symplectic toric manifolds;
in Section \ref{sec:density} we review the definition of the density function on $\Mo$;
in Section \ref{sec:convex} we explain how to reduce the proof of Theorem \ref{symplectic-geometry}
to the proof of a theorem in convex geometry (Theorem \ref{convex-geometry});
in Section \ref{sec:proof} we state and prove Theorem \ref{convex-geometry}.
\textup{\,}
\\
\\
\emph{Acknowledgements}. This work was supported by NSF and a J Tinsley Oden Fellowship at ICES, UT Austin. AP  thanks Luis Caffarelli and Alessio Figalli for sponsoring his visit.

\section{Main theorem} \label{sec:0}

 Let $\T$ be a $2$\--dimensional torus, 
let $\Mo$ be the moduli space of symplectic toric $4$\--manifolds. 
It is known \cite{PePiRaSa} that  $\mathcal{M}_{\T}$ is a neither a locally compact
nor a complete metric space but its completion is well understood and describable in
explicit terms (see Theorem~\ref{pp}). Let $\Omega \colon \Mo \to [0,1]$ be the 
\emph{optimal density function}, which assigns to a manifold the density of its optimal toric ball packing, cf.
Figure~\ref{optimalpacking}.  
By construction $\Omega$ \emph{is an invariant} of the symplectic toric type of $M$.
 It is natural to wonder what the precise regions of continuity of $\Omega$ are; this problem
 was posed in \cite[Problem 30]{PePiRaSa}. The following result solves this problem by
 giving a characterization of the regions where $\Omega$ is continous.

\begin{theoremL} \label{symplectic-geometry}
Let $N \geq 1$ be an integer and let $\Mo^N$ be the set of symplectic toric manifolds with precisely $N$ points fixed
by the $\T$\--action. Then:
\begin{itemize}
\item[{\rm (1)}]
$\Omega$ is discontinuous at every $(M,\omega,\T) \in \Mo$, and the restriction 
$\Omega|_{\Mo^N}$ is continuous for each $N \geq 1$.
\item[{\rm (2)}]
Given $(M,\omega,\T) \in \Mo^N$, define $\Omega_i (M,\omega,\T)$, where $1 \leq i \leq N$, to be the
optimal density computed along all packings avoiding balls with center at the $i$th fixed point of 
the $\T$\--action. Then $\Mo^N$ is the largest neighborhood of $M$
in $\Mo$ where  $\Omega$ is continuous if and only if $\Omega_i(M,\omega,\T)<\Omega(M,\omega,\T)$ for all
$i$ with $1\leq i \leq N$.
\end{itemize}
\end{theoremL}

Note that $$\Mo=\bigcup_{N\geq 1} \Mo^N.$$

The aforementioned classification of symplectic toric manifolds is 
 in terms of a class of convex polytopes in
$\mathbb{R}^n$, called now \emph{Delzant polytopes}.  In practice this means that the a priori
very complicated packing problem for these systems can be formulated and studied in terms of convex
geometry, which is what was done in  \cite{Pe1,Pe2,PeSc2008}. 
This is because symplectic toric manifolds have an associated momentum map $\mu \colon M \to \mathbb{R}^n,$
which is a Morse\--Bott function with a number of remarkable properties and which captures all the main features
of the symplectic geometry of $(M,\omega)$ and of the Hamiltonian 
action of $\T$ on $M$ (the proof of this is due
to Atiyah, Guillemin, Sternberg, and Delzant). In fact $\mu(M)$ is a convex polytope in $\mathbb{R}^n$.

\begin{figure} \label{optimalpacking}
\begin{center}
\epsfig{file=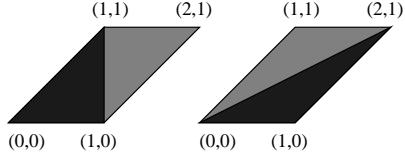}
\caption{Two distinct optimal packings of a symplectic toric  manifold $M$ isomorphic to $\mathbb{CP}^1 \times \mathbb{CP}^1$, so $\Omega(M)=1$.}
\end{center}
\end{figure}

{The strategy of this paper is to prove a slightly more general result about polytopes,} Theorem \ref{convex-geometry},
which implies Theorem~\ref{symplectic-geometry} by virtue of a theorem proven in \cite{Pe2} (reviewed in the present paper, see Theorem~\ref{2:t}).

\begin{remark}
\normalfont
The proof of 
Theorem~\ref{convex-geometry} is based on simple geometric arguments that can be easily understood in the case $n=2$
(i.e., when dealing with Delzant polygons) but the general approach in the proof should work  in any dimension
(indeed, in the proof we rely only on results from \cite{PeSc2008} that hold in any dimension). 
However the case of symplectic toric $4$\--manifolds
is currently the most interesting as the topology of the moduli space is understood only in that case
(see Theorem \ref{pp}).
\end{remark}

\section{Symplectic toric manifolds} \label{sec:st}
In this section we review the notion of symplectic toric manifold in arbitrary dimension. 

Let $(M, \,\omega)$ be a closed symplectic $2n$\--dimensional manifold.  Let $\mathbb{T}^k\cong (S^1)^k$ be the $k$\--dimensional torus,
and write $\mathbb{T}:=\mathbb{T}^n$ to denote 
the standard $n$\--dimensional torus.

Let $\mathfrak{t}$ be the Lie algebra $\mbox{Lie}(\T)$ 
of $\T$ and let $\mathfrak{t}^*$ be the dual of $\mathfrak{t}$.  
A symplectic action $\psi: \mathbb{T}^k \times M \rightarrow M$ of a 
$k$\--dimensional torus (that is, an action preserving the form $\omega)$ 
is called \emph{Hamiltonian}  if there is a map $$\mu: M \rightarrow \frak{t}^*,$$
known as a \emph{momentum map}, satisfying Hamilton's equation 
$$\textup{i}_{\xi_M} \omega=\textup{d} \langle \mu,\, \xi \rangle,$$ for all $\xi \in \frak{t}$. The momentum map is defined up to translation by an element of $\frak{t}^*$.  Nevertheless, we ignore this 
ambiguity and call it \emph{the momentum map}.   

\begin{definition}
A \emph{symplectic\--toric manifold} is a quadruple $(M,\,\omega, \T,\mu)$ where $M$ is  a $2n$\--dimensional 
closed symplectic manifold $(M,\, \omega)$ equipped with an effective Hamiltonian action 
 of an $n$\--dimensional torus $\T$ 
with momentum map $\mu$. 
\end{definition}

\begin{example}  
\normalfont
The projective space $(\mathbb{CP}^n,\, \lambda \cdot \omega_{\textup{FS}})$, where  
$\omega_{\textup{FS}}$ is the Fubini\--Study form given by
$$
\omega_{\textup{FS}}=\frac{1}{2 (\sum_{i=0}^n \bar{z}_i z_i)} \sum_{k=0}^n \sum_{j \neq k} (\bar z_j z_j \, \textup{d}z_k \wedge \textup{d} \bar{z}_k-
\bar z_j z_k  \, \textup{d}z_j \wedge \textup{d} \bar{z}_k)
$$
equipped with the rotational action of $\mathbb{T}^n$, 
$$
(\textup{e}^{\textup{i} \theta_1},\ldots,\textup{e}^{\textup{i} \theta_n}) \cdot [z_0: \ldots:z_n]=[z_0:\, \textup{e}^{-2 \pi \textup{i} \theta_1}\, z_1: \ldots:\textup{e}^{-2 \pi \textup{i} \theta_n} \, z_n],
$$
is a $2n$\--dimensional symplectic\--toric manifold. The components of the momentum map are
$$\mu^{\mathbb{CP}^n}_k(z)=\frac{\lambda |z_k|}{\sum_{i=0}^n|z_i|^2}.$$
The corresponding momentum polytope is equal to  the convex hull in $\mathbb{R}^n$ of $0$ and 
the scaled canonical vectors $\lambda e_1,\ldots,\lambda e_n$.
\end{example}

Strictly speaking,  $\mu$ is a map from $M$ to 
$\mathfrak{t}^*$.  However, the presentation is simpler if 
from the beginning we identify both $\mathfrak{t}$ and $\mathfrak{t}^*$ with $\mathbb R^n$
and consider $\mu$ as a map 
from $M$ to $\mathbb{R}^n$. The procedure to do this, that we now describe, is standard
but not canonical. Choose an epimorphism 
$E \colon \mathbb{R} \to \mathbb{T}^1$, for instance, 
$x \mapsto e^{2\pi\, \sqrt{-1}x}$. This Lie group epimorphism 
has discrete center $\mathbb{Z}$ and the inverse of the  
corresponding Lie algebra isomorphism is given by
$\mbox{Lie}(\mathbb{T}^1) \ni \frac{\partial}{\partial x} 
\mapsto \frac{1}{2\pi}\in \mathbb{R}$. Thus, for
$\mathbb{T}$ we get the 
non-canonical isomorphism between the corresponding
commutative Lie algebras
$$
\mbox{Lie}(\mathbb{T}) =
\mathfrak{t} \ni \frac{\partial}{\partial x_k} \longmapsto 
\frac{1}{2\pi} \, {\rm e}_k \in\mathbb{R}^n,
$$ 
where ${\rm e}_k$ is the $k^{\rm th}$ element in the canonical basis
of $\mathbb{R}^n$. Choosing an inner product 
$\langle \cdot,\cdot \rangle$ on $\mathfrak{t}$, we obtain an 
isomorphism $\mathfrak{t} \to \mathfrak{t}^*$, and hence 
taking its inverse and composing it with the isomorphism
$\mathfrak{t} \to \mathbb{R}^n$ described above, we get an 
isomorphism $\mathcal{I}: \mathfrak{t}^* \to \mathbb{R}^n$. 
In this way, we obtain a momentum map $\mu=\mu_{\mathcal{I}} \colon M \to \mathbb{R}^n$.

\begin{example}
\normalfont
Consider the open ball $\mathbb{B}_{r}$ if radius $r$ in $\subset \mathbb{C}^n$, equipped 
with the standard symplectic form $\omega_0=\frac{i}{2}\, \sum_{j} \textup{d}z_j \wedge 
\textup{d}\overline{z_j}$ and the Hamiltonian action ${\rm Rot}$ by rotations given by 
$(\theta_1,\ldots, \theta_n)\cdot(z_1,\ldots, z_n)=
(\theta_1 \, z_1, \ldots ,\theta_n \, z_n)
$.  In this case the components of the  momentum map $\mu^{\mathbb{B}_{r}}$ are 
$\mu^{\mathbb{B}_{r}}_k=|z_k|^{2}$.  The image $\Delta^n(r)$ of the momentum map is 
\begin{equation} \label{stsimplex}
\Delta^n(r)=\ConvHull(0,r^{2}\, e_1,\ldots, r^{2}\, e_n)\setminus \ConvHull(r^{2}\, e_1,\ldots, r^{2}\, e_n),
\end{equation}
where $\{e_i\}_{i=1}^n$ stands for the canonical basis of $\mathbb{R}^n$. 
\end{example}

 The image of the momentum map of a  symplectic\--toric manifold is  a particular class of convex polytope, a so called \emph{Delzant polytope}, see
 Figure~\ref{xxx}.

\begin{figure}[htb]
\begin{center}
\epsfig{file=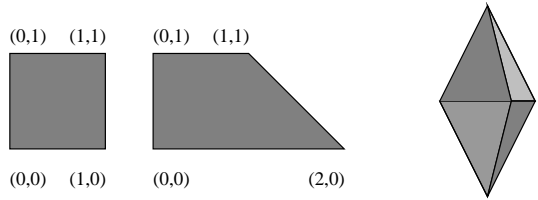}
\label{xxx}
\caption{Two Delzant polytopes (left) and a non-Delzant polytope (right).}
\end{center}
\end{figure}

\begin{definition}\label{del2} {\rm (\cite{anacannas})}
A convex polytope $\Delta$ in $\R^n$ is a \emph{Delzant polytope} if it is simple, rational and smooth:
\begin{itemize}
\item[{\rm (i)}] $\Delta$ is \emph{simple} if there are exactly $n$ edges meeting at each vertex $v\in V$;
\item[{\rm (ii)}] $\Delta$ is \emph{rational} if for every vertex $v\in V$, the edges meeting at $v$ are of the form $v+t u_i$, $t\geq 0$, and $u_i\in \Z^n$;
\item[{\rm (iii)}] A vertex $v\in V$ is \emph{smooth} if the edges meeting at $v$ are of the form $v+tu_i$, $t\geq 0$, where the vectors
$u_1,\ldots,u_n$ can be chosen to be a $\Z$ basis of $\Z^n$. $\Delta$ is \emph{smooth} if every vertex $v\in V$ is smooth.
\end{itemize}
\end{definition}

\section{Moduli spaces of  toric manifolds} \label{sec:moduli}

With the conventions above, where 
$\T$ and the identification $\mathcal{I} \colon
\mathfrak{t}^* \to \mathbb{R}^n$ are \emph{fixed}, we next 
define the moduli space of toric manifolds.   

\subsection{The moduli relation}

Let $$(M,\omega,\T,\mu \colon M \to \mathbb{R}^n)$$ and 
$$(M',\omega',\T,\mu' \colon M \to \mathbb{R}^n)$$ be symplectic 
toric manifolds. 
These two
symplectic toric manifolds are \emph{isomorphic} if 
there exists an equivariant symplectomorphism  
$\varphi\colon M\to M'$ (i.e., $\varphi$ is a diffeomorphism 
satisfying $\varphi^\ast \omega' = \omega$ which intertwines the $\T$ actions) such that 
$\mu'\circ \varphi=\mu$ (see also \cite[Definition I.1.16]{AuCaLe}). 
We denote by $$\Mo:=\Mo^{\mathcal{I}}$$ 
the \emph{moduli space} (that is, the set of equivalence classes) of 
symplectic toric manifolds of a fixed dimension $2n$ under this 
equivalence relation. 

\begin{thm}{\rm (Delzant \cite[Theorem 2.1]{De1988})}\label{Del1}
Let $(M,\omega,\T,\mu)$ and $(M',\omega',\T,\mu' )$ be symplectic toric 
manifolds. If the images $\mu(M)$ and $\mu'(M')$ are equal, then $(M,\omega,\T,\mu)$ and 
$(M',\omega',\T,\mu' )$ are isomorphic.
\end{thm}

The convexity theorem of Atiyah \cite{At1982} and Guillemin\--Sternberg \cite{GuSt1982} says that
the momentum map image is a convex polytope. In addition, if the dimension of the torus is precisely
half the dimension of the manifold, this polytope is Delzant (Definition~\ref{del2}).

Let $\mathcal{D}_\mathbb{T}$ denote the set of Delzant
polytopes.  It follows from Theorem \ref{Del1} that
\begin{equation}
\label{map}
\left[ (M,\omega,\T,\mu) \right] \ni \mathcal{M}_\mathbb{T}
\longmapsto \mu(M) \in \mathcal{D}_\mathbb{T}
\end{equation}
is injective. Delzant also showed how starting from a Delzant
polytope it is possible to reconstruct a symplectic toric manifold with momentum
image precisely equal to the Delzant polytope, thus implying that
\eqref{map} is a \emph{bijection}. To simplify our notation we often write 
$(M,\omega,\T,\mu)$ identifying the representative with
the corresponding equivalence class $\left[(M,\omega,\T,\mu) \right]$
in the moduli space $\mathcal{M}_\mathbb{T}$.

\subsection{Metric on $\Mo$}

Endow $\Dt$ with the distance function 
given by the volume of the symmetric difference 
$$
(\Delta_1\smallsetminus \Delta_2) \cup 
(\Delta_2\smallsetminus \Delta_1)
$$ 
of  any two polytopes $\Delta_1$ and $\Delta_2$. Using (\ref{map}) we define a metric $d_\T$ on 
$\Mo$ as the pullback of the metric defined on $\Dt$. In this way we
obtain the metric space $(\Mo,d_\T)$. This metric 
induces a topology $\nu$ on $\Mo$.

Consider  the $\sigma$-algebra $\mathfrak{B}(\R^n)$ of Borel sets 
of $\R^n$. Let $\lambda\colon\mathfrak{B}(\R^n)\to\R_{\geq 0} 
\cup \{\infty\}$ be the Lebesgue measure on $\R^n$, and let 
$\mathfrak{B}'(\R^n)\subset \mathfrak{B}(\R^n)$
consist of the Borel sets with finite Lebesgue measure.
Let  $\chi_C\colon \R^n\to \R$ be the characteristic function of $C\in \mathfrak{B}'(\R^n)$. 
Define  
\begin{eqnarray} \label{first}
 d(A,B):=\left\|\chi_A-\chi_B\right\|_{{\rm L}^1}.
\end{eqnarray}
This extends the distance function defined above on $\Dt$. However, it is not a metric on $\mathfrak{B}'(\R^n)$. 
By identifying the sets $A, B \in \mathfrak{B}'(\R^n)$ with
$d(A,B)=0$, we obtain a metric on the resulting quotient space of $\mathfrak{B}'(\R^n)$.  
Let $\mathcal{C}$ be the set of convex compact subsets of 
$\R^n$ with positive Lebesgue measure,
$\varnothing$ the empty set, and
$\hat{\mathcal{C}}:=\mathcal{C}\cup \{\varnothing\}.$ 
Then $\hat{\mathcal{C}}$ equipped with $d$ in (\ref{first}) is a metric space.

 \begin{remark}{\rm (Other moduli spaces)}
 \normalfont
 There is no
 ${\rm AGL}(n,\Z)$ equivalence relation involved in the definition of $\mathcal{D}_{\mathbb{T}}$.
 Such equivalence relation is often put on this space so that it is in one\--to\--one correspondence with the
 moduli space of symplectic toric manifolds \emph{up to equivariant isomorphisms}. This 
 is \emph{not} the relation which is relevant to this paper. Nonetheless, to avoid confusion let us recall 
  that 
 two symplectic toric manifolds $(M,\omega,\T,\mu)$ and $(M',\omega',\T,\mu')$ are \emph{equivariantly isomorphic} if there exists an automorphism of the torus $h\colon \T\to \T$ and an $h$-equivariant symplectomorphism $\varphi\colon M\to M'$, i.e., such that the following diagram commutes:
\begin{equation}\label{comm action1} \nonumber
\xymatrix{
\T\times M \;\; \ar[r]^{\;\;\rho*} \ar[d]_{(h,\varphi)} & M \ar[d]_{\varphi} \\
\T \times M' \;\;\ar[r]^{\;\;\rho'^*} &  M'.
 } \
\end{equation}
In \cite{PePiRaSa},  the space $\Mt$ denotes the \emph{moduli space of equivariantly isomorphic $2n$-dimensional
symplectic toric manifolds.}\footnote{ 
Two equivariantly isomorphic toric manifolds $(M,\omega,\T,\mu)$ and $(M',\omega',\T,\mu')$ 
are isomorphic if and only if $h$ in (\ref{comm action1}) is the identity and $\mu'=\mu \circ \varphi$.} 
Let $\agl(n,\Z)=\operatorname{GL}(n,\Z)\ltimes \R^n$ be the group 
of affine transformations of $\R^n$ given by $x\mapsto Ax+c,$ where
$A\in \operatorname{GL}(n,\Z)$ and $c\in \R^n$.
We say that two Delzant polytopes $\Delta_1$ and $\Delta_2$
are \emph{$\agl(n,\Z)$\--equivalent} if there exists $\alpha\in \agl(n,\Z)$
such that $\alpha(\Delta_1)=\Delta_2$.
Let $\Dtt$ be the moduli space of 
Delzant polytopes modulo the equivalence relation given by 
$\agl(n,\Z)$; we endow this space with the quotient topology induced by the projection map
$\pi\colon \Dt \to \Dtt\simeq \Dt/\agl(n,\Z)$. The map (\ref{map})
induces a bijection between $\Mt$ and $\Dtt$. Thus $\Mt$ is  a topological space with the topology $\widetilde{\nu}$ 
induced by this bijection. The topological space $(\Mt,\widetilde{\nu})$ is studied in \cite{PePiRaSa}. It would be interesting to give a version of
Theorem~\ref{symplectic-geometry} where $\Mo$ is replaced by $\Mt$. It was proven in \cite[Theorem~1]{PePiRaSa}
that $(\Mt,\widetilde{\nu})$  is connected.

 \end{remark}
 
 \subsection{Topological structure in dimension $4$}
 
 The following theorem, \emph{which holds in dimension $4$}, describes properties of the 
 topology of the aforementioned moduli space.
 
\begin{thm}[\cite{PePiRaSa}] \label{pp}
Let $\T$ be a $2$\--dimensional torus. The space 
$(\Mo,d_\T)$ is neither locally compact nor 
a complete metric space. Its completion can be identified with 
the metric space $(\hat{\mathcal{C}},d)$ in the following 
sense: identifying $(\Mo,d_\T)$ with $(\Dt,d)$ via 
{\rm (\ref{map})}, the completion of  $(\mathcal{D}_{\T},d)$ 
is $(\hat{\mathcal{C}},d)$.  
\end{thm}

\section{Density function} \label{sec:density}

Let $(M,\, \omega,\T,\mu)$ be a symplectic toric manifold of dimension $2n$.
Let $\Lambda$ be an automorphism of $\mathbb{T}$. Let $r>0$. 
We say that a subset $B$ of $M$ is a \emph{$\Lambda$\--equivariantly embedded symplectic ball 
of radius $r$} if there is a symplectic embedding $f:\mathbb{B}_{r} \rightarrow M$ with $f(\mathbb{B}_r)=B$ and such that the diagram:
\begin{eqnarray} 
\xymatrix{ \ar @{} [dr] |{\circlearrowleft}
\mathbb{T}^{n} \times \mathbb{B}_{r}  \ar[r]^{\Lambda \times f}      \ar[d]^{ \textup{Rot} }  &  \mathbb{T}^{n} \times M 
                 \ar[d]^{\psi}   \\
                   \mathbb{B}_{r}  \ar[r]^f   &       M    } \nonumber
\end{eqnarray} 
is commutative. Let $B$ be a $\Lambda$\--equivariantly embedded symplectic ball. We say that 
$B$ \emph{has center $f(0)$}. We say that another 
subset $B'$ of $M$ is an \emph{equivariantly embedded symplectic ball
of radius $r'$} if there exists an automorphism $\Lambda'$ of $\mathbb{T}$ such that
$B'$ is a $\Lambda'$\--equivariantly embedded symplectic ball of radius $r'$.

In what follows, the symplectic volume of a subset $X\subset M$ is given by 
$$
\vol_{\omega}(X):=\int_{X} \omega^{n}.
$$ 
Following \cite[Definition~1.6]{Pe2}, a \emph{toric ball packing of $M$} 
is given by a disjoint union of the form $$\mathcal{P}:=\bigsqcup_{\alpha \in A} B_{\alpha},$$ 
where each $B_{\alpha}$ is an  equivariantly embedded 
symplectic ball (note that we are not saying that all these balls
must have the same radii).  The 
\emph{density} $\Omega(\mathcal{P})$ of a toric ball packing $\mathcal{P}$ is given by the quotient
$$
\Omega(\mathcal{P}):= \frac{\vol_{\omega}(\mathcal{P})}{\vol_{\omega}(M)} \in [0,1]. 
$$
The \emph{density of a symplectic\--toric manifold} $(M,\, \omega,\T,\mu)$ is given by 
$$\Omega(M,\, \omega,\T,\mu):= \sup \Big\{\, \Omega(\mathcal{P})\, | \, \mathcal{P} \, \textup{is a toric ball packing of}\, M \Big\}.$$  
An \emph{optimal packing} (also called a \emph{maximal packing}) is a toric ball packing at which this density is achieved.
The \emph{optimal density function}  is defined as follows: it assigns to a manifold the density of one of its 
optimal packings (such an optimal packing always exists,
see \cite{PeSc2008}).
The optimal density function is interesting to us because \emph{it is a symplectic invariant}. 
In this paper we analyze the continuity of the density
function on the moduli space $\mathcal{M}_{\mathbb{T}}$  (see Theorem \ref{symplectic-geometry}), which was posed as an open problem in
 \cite[Problems 4 and 30]{PePiRaSa}.

\section{Convex geometry} \label{sec:convex}

Following \cite{Pe2, PeSc2008} we say that a subset $\Sigma$ of a Delzant polytope $\Delta$ is 
\emph{an admissible simplex of radius $r$ with center at a vertex $v$ of $\Delta$} if 
there is an element of $\AGL(n,\,\mathbb{Z})$ which takes:
\begin{itemize}
\item
 $\Delta(r^{1/2})$ to $\Sigma$,
 \item
 the origin to $v$,
 \item
 the edges of $\Delta(r^{1/2})$ meeting at the origin to  the edges of $\Delta$ meeting at $v$. 
\end{itemize}
We write
$$r_v:=\max\Big\{r>0\, |\, \exists \, \textup{an admissible simplex of radius $r$ with center $v$}\Big\}$$
for every vertex $v$.

In what follows we write  $\vol_{\textup{euc}}(A)$ for the Euclidean volume of a subset $A\subset \Delta$.

\begin{figure}[htb]
\begin{center}
\includegraphics{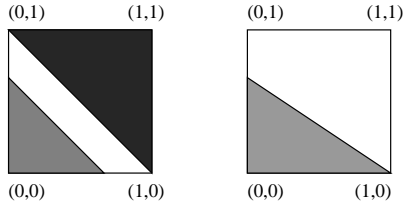}
\caption{The packing on the left is admissible, the one on the right is not.}\label{AF4}
\end{center}
\end{figure}

An \emph{admissible packing of a Delzant polytope $\Delta$} is given by a
a disjoint union 
$$\mathcal{P}:=\bigsqcup_{\alpha \in A} \Sigma_{\alpha},$$
where each $\Sigma_{\alpha}$ is 
an admissible simplex (note that we are not saying that they must have the same radii).  
The \emph{density} of an admissible packing $\mathcal{P}$ 
is  defined by the quotient 
$$\Omega(\mathcal{P}):= \frac{\vol_{\textup{euc}}(\mathcal{P})}{\vol_{\textup{euc}}(\Delta)} \in [0,\,1],$$ 
and the \emph{density of $\Delta$}
is defined by $$\Omega(\Delta):= \sup \Big\{\, \Omega(\mathcal{P})\, | \, \mathcal{P} \, \textup{is an admissible packing of}\, \Delta \Big\}.$$  
An \emph{optimal packing} (also called a \emph{maximal packing}) is an admissible ball packing at which this density is achieved. 
The \emph{optimal density function}  assigns to a Delzant polytope the density of one of its 
optimal packings (such packing always exists,
see \cite{PeSc2008}).
 
Figure~\ref{AF4} shows an example of an admissible packing and a non admissible one; at the manifold
level, the admissible one corresponds to a symplectic toric ball packing of $\mathbb{CP}^1 \times \mathbb{CP}^1$ by
two disjoint balls, while the non admissible one cannot be interpreted in this way because the shaded
triangle does not correspond to an equivariantly and symplectially embedded ball. Both packings in
Figure~\ref{AF5} are admissible.
 
\begin{figure}[htb]
\begin{center}
\includegraphics{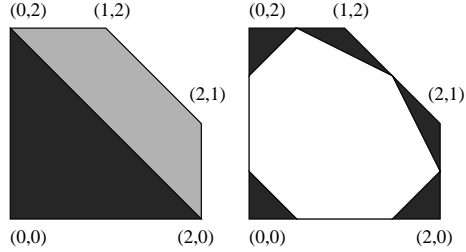}
\caption{Admissible packings of the polygon of a blow up of the product of
two symplectic spheres of radius $1/\sqrt{2}$.}\label{AF5}
\end{center}
\end{figure}

  Let $I\subset \mathbb{R}^n$ be an interval with rational slope.
  The \emph{rational\--length} $\length_{\mathbb{Q}}(I)$ of $I$ is the (unique) number $\ell$  such that 
  $I$ is $\AGL(n,\, \mathbb{Z})$\--congruent to a length $\ell$ interval on a coordinate axis.  
  As it is shown by the following result, an admissible simplex is parametrized by its center
 and its radius.

 \begin{lemma}[\cite{Pe2}] \label{adsi}
 Let $\Delta$ be a Delzant polytope and let $v\in \Delta$ be a vertex of $\Delta$.
  We denote the $n$ edges leaving $v$ by $e_v^1,\ldots, e_v^n$.  
  Then:
  \begin{itemize}
  \item[(i)]
  $r_v$ is given by
 $$r_v=\min\Big\{\length_{\mathbb{Q}}(e_{v}^1), \ldots, \length_{\mathbb{Q}}(e_{v}^n) \Big\}.
 $$  
 \item[(ii)]
 There exists an admissible simplex $\Sigma(v,\,r)$ centered at $v$ of radius $r$ if and only if we have that $0\leq r \leq r_v$.  Moreover $\Sigma(v,\,r)$ is the unique such admissible simplex.
 \item[(iii)]
 The volume of $\Sigma(v,\,r)$ in (ii) is given by
   $$\vol_{\textup{euc}}(\Sigma(v,\, r))= \frac{r^{n}}{n!}.
   $$
 \end{itemize}
 \end{lemma}
  
  Now we can state the tool that we will use in the proof of the main theorem of the paper.
  Consider a symplectic\--toric manifold $(M, \, \omega,\, \T,\mu)$ and let $\Delta:=\mu(M)$.
   
 \begin{thm}[\cite{Pe2}] \label{2:t}  
 The following hold:
\begin{itemize}
\item[{\rm (i)}]
If  $B$ is an 
equivariantly embedded symplectic ball with center at $p \in M$ and 
radius $r$, the momentum map image  $\mu(B)$ is an admissible 
simplex in $\Delta$ centered at $\mu(p)$ of radius $r^{2}$.  

Moreover, if $\Sigma$ is an admissible simplex 
in $\Delta$ of radius $r$, then there is an equivariantly embedded symplectic ball $B$ in $M$ of radius $r^{1/2}$ 
and such that $$\mu(B)=\Sigma.$$
\item[{\rm (ii)}]
 Let $\mathcal{P}$ be a toric ball packing of $M$. Then $\mu(\mathcal{P})$ is an admissible packing of $\Delta$.
 Furthermore, we have the equality 
 $$\Omega(\mathcal{P})=\Omega(\mu(\mathcal{P})).
 $$  
 
 Moreover, for any admissible packing $\mathcal{Q}$ of $\Delta$ there is a toric ball packing $\mathcal{P}$ of $M$ such that $$\mu(\mathcal{P})=\mathcal{Q}.$$ 
 \end{itemize}
 \end{thm}

\section{The combinatorial convexity statement}
\label{sec:proof}

In this section we state and prove a theorem in convex geometry, that in view of Theorem \ref{2:t} implies Theorem \ref{symplectic-geometry}.

\begin{thm} \label{convex-geometry}
Let $N \geq 1$ be an integer and let $\mathcal{P}^N$ be the set of Delzant polygons of $N$ vertices, and let
$\mathcal{P}$ be the set of all Delzant polygons, so that $\mathcal{P}=\bigcup_{N\geq 1} \mathcal{P}^N.$ Then:
\begin{itemize}
\item[(1)]
$\Omega$ is discontinuous at every $\Delta \in \mathcal{P}$, and the restriction 
$\Omega|_{\mathcal{P}^N}$ is continuous for each $N \geq 1$.
\item[(2)]
Given $\Delta \in \mathcal{P}^N \subset \mathcal{P}$, define $\{\Omega_i(\Delta)\}_{1 \leq i \leq N}$ to be the
maximal density avoiding vertex $i$. Then $\mathcal{P}^N$ is the largest set containing  $\Delta$
where $\Omega$ is continuous if and only if $\Omega_i(\Delta)<\Omega(\Delta)$ for all
$1\leq i \leq N$.
\end{itemize}
\end{thm}

\begin{proof}
In the proof of this result we will constantly make use of the following simple observation (see also the discussion below): 
given $\Delta \in \mathcal{P}$, a neighborhood of $\Delta$ is made up of polygons where either we translate
the sizes in a parallel way, or we chop
some small corners of $\Delta$. In particular the number of vertices can only increase.\\

We prove (1) first. To show that
$\Omega$ is discontinuous at any $\Delta \in \mathcal{P}^N$, we fix $\epsilon>0$ small and we chop all
the corners, adding around each of them a small side of length $\leq \epsilon$, making sure that the new polygon is still a Delzant polygon.

More precisely, if two consecutive sides are given by $(0,a)$ and $(b,0)$ for some $a,b>0$, because the polygon is a Delzant polygon it follows that $a/b \in \mathbb Q$. Hence, if we chop the corner at the origin with the segment connecting $(a/M,0)$ to $(0,b/M)$
where $M \in \mathbb N$ is very large, we obtain a new Delzant polygon.
For the general case, given a couple of consecutive sides,
there is a map in ${\rm AGL}(2,\mathbb Z)$ sending
$(0,0)$ to the common vertex, and 
$(a,0)$ and $(0,b)$ to the other two vertices. Hence, it suffices to chop from the polygon the image
of the triangle determined by $(a/M,0)$, $(0,b/M)$, $(0,0)$.

This construction gives us a polygon with $2N$ vertices with the following property: when considering the definition of density,
any simplex will have at least one side of length at most $\epsilon$ while the others are universally bounded,
so the total volume will be at most $CN\epsilon$. Since $\epsilon$ can be arbitrarily small, this proves that $\Omega$
cannot be continuous.

On the other hand, we show that $\Omega$ is continuous when restricted to $\mathcal{P}^N$.
To see this, we call an angle $\alpha$ smooth if it can be obtained as the angle of  a smooth triangle
having the origin as one of its vertices,
$\alpha$ being the angle at the origin.
Since the other two vertices belong to $\mathbb Z^2$ (call them $v_1,v_2$), it follows that 
the sides $(0,v_1)$ and $(0,v_2)$ have lengths $\ell_1,\ell_2 \geq 1$.
Also, since such triangle is the image of $\Delta^2(1)$ under a map in $\AGL(2,\mathbb Z)$
its area must be $1/2$, that is
$$\frac{1}{2}=\frac{\ell_1\ell_2 \sin \alpha}{2}.$$
These facts imply that there exists a constant $C_\alpha$, depending only on $\sin\alpha$, such that
 $$\ell_1,\ell_2 \leq C_\alpha.$$
 
Now, fix $\Delta \in \mathcal{P}_N$, and let $\alpha_1,\ldots, \alpha_N$ be the angles of $\Delta$. 
Then, if $\Delta' =(\alpha'_1,\ldots,\alpha'_N)\in \mathcal{P}^N$, 
since the set of smooth angles is discrete it follows that $\alpha'_i=\alpha_i$
whenever $d(\Delta',\Delta) \ll1$.
Hence, for every $\Delta'$ close to $\Delta$ we are just translating the sides of the polygon in a parallel way, and since $\Omega$ is
continuous along this family of transformations \cite{PeSc2008} this proves that $\Omega$ is continuous on the whole $\mathcal{P}^N$, proving (1).\\

Next we show (2). 
Assume first that $$\Omega(\Delta)=\Omega_i(\Delta)\,\,\, \textup{for some}\,\,\, i \in \{1,\ldots,N\}.$$ 
This implies that we can find an optimal family
in the definition of $\Omega$ with no simplex centered at $i$ (simply take an optimal family in the definition of $\Omega_i(\Delta)$). 
Then, by slightly reducing the radius of each simplex
we can keep the density arbitrarily close to $\Omega(\Delta)$ making sure that no simplex touches the vertex $i$.
This gives us the possibility to chop a small corner around the vertex $i$, obtaining  Delzant polytope $\Delta' \in \mathcal P^{N+1}$ keeping the density still arbitrarily close  $\Omega(\Delta)$.
This procedure shows that we can find a family $\{\Delta_k\}_{k \geq 1} \subset \mathcal P^{N+1}$ with 
\begin{equation}
\label{eq:delta k continuous}
d(\Delta_k,\Delta)\to 0\,\,\, \textup{and} \,\,\, \Omega(\Delta_k)\to \Omega(\Delta),
\end{equation} proving that there is a set larger than
$\mathcal{P}^N$ where $\Omega$ is continuous at $\Delta$.
Viceversa,  assume there exists $\eta>0$ such that $$\Omega_i(\Delta) \leq \Omega(\Delta)-\eta\,\,\, \textup{for all} \,\,\,i=1,\ldots,N.$$
Since it is impossible to approximate $\Delta$ with polygons $\Delta'$ that have strictly less than $N$ vertices
and we already proved that $\Omega$ is continuous on $\mathcal P^N$, 
we have to exclude that we can find some sequence $\{\Delta_k\}_{k\in \mathbb N}\subset \cup_{j>N}\mathcal P^j$
such that \eqref{eq:delta k continuous} hold.

To see this notice that 
if we chop any corner around one of the vertices (say around $i$) adding a small side,
then necessarily the density will be close to $\Omega_i(\Delta)$, hence less than $$\Omega(\Delta)-\eta/2,$$
which proves that  \eqref{eq:delta k continuous} cannot hold. This concludes the
proof of (2).
\end{proof}

\bibliographystyle{new}

\newpage

\noindent
\\
{\bf Alessio Figalli}\\
The University of Texas at Austin\\
Mathematics Dept. RLM 8.100\\
2515 Speedway Stop C1200\\
Austin, Texas 78712-1202, USA\\
{\em E\--mail}: {figalli@math.utexas.edu}\\

\noindent
 {\bf {\'A}lvaro Pelayo} \\
  Department of Mathematics\\
University of California, San Diego\\
9500 Gilman Drive  \# 0112\\
La Jolla, CA 92093-0112, USA\\
{\em E\--mail}: {alpelayo@math.ucsd.edu}\\

\end{document}